\documentclass[11pt]{article}
\makeatletter

\usepackage{amssymb,amsmath,amsthm,latexsym}
\title{Solution to an isotopism question concerning rank 2 semifields}
\author{Michel Lavrauw, Giuseppe Marino, Olga Polverino and Rocco Trombetti}
\date{}

\begin{document}
\maketitle

\newtheorem{theorem}{Theorem}[section]
\newtheorem{lemma}[theorem]{Lemma}
\newtheorem{conj}[theorem]{Conjecture}
\newtheorem{remark}[theorem]{Remark}
\newtheorem{cor}[theorem]{Corollary}
\newtheorem{prop}[theorem]{Proposition}
\newtheorem{defin}[theorem]{Definition}
\newtheorem{result}[theorem]{Result}
\newtheorem{property}[theorem]{Property}

\makeatother
\newcommand{\Prf}{\noindent{\bf Proof}.\quad }
\renewcommand{\labelenumi}{(\alph{enumi})}


\def\B{\mathbf B}
\def\C{\mathbf C}
\def\Z{\mathbf Z}
\def\Q{\mathbf Q}
\def\W{\mathbf W}
\def\a{\mathbf a}
\def\b{\mathbf b}
\def\c{\mathbf c}
\def\d{\mathbf d}
\def\e{\mathbf e}
\def\l{\mathbf l}
\def\v{\mathbf v}
\def\w{\mathbf w}
\def\x{\mathbf x}
\def\y{\mathbf y}
\def\z{\mathbf z}
\def\t{\mathbf t}
\def\cD{\mathcal D}
\def\cC{\mathcal C}
\def\cH{\mathcal H}
\def\cM{{\mathcal M}}
\def\cK{\mathcal K}
\def\cQ{\mathcal Q}
\def\cU{\mathcal U}
\def\cS{\mathcal S}
\def\cT{\mathcal T}
\def\cR{\mathcal R}
\def\cN{\mathcal N}
\def\cA{\mathcal A}
\def\cF{\mathcal F}
\def\cL{\mathcal L}
\def\cP{\mathcal P}
\def\cG{\mathcal G}
\def\cGD{\mathcal GD}

\def\PG{{\rm PG}}
\def\GF{{\rm GF}}

\def\Pg{PG(5,q)}
\def\pg{PG(3,q^2)}
\def\ppg{PG(3,q)}
\def\HH{{\cal H}(2,q^2)}
\def\F{\mathbb F}
\def\Ft{\mathbb F_{q^t}}
\def\P{\mathbb P}
\def\V{\mathbb V}
\def\bS{\mathbb S}
\def\G{\mathbb G}
\def\E{\mathbb E}
\def\N{\mathbb N}
\def\K{\mathbb K}
\def\D{\mathbb D}
\def\ps@headings{
 \def\@oddhead{\footnotesize\rm\hfill\runningheadodd\hfill\thepage}
 \def\@evenhead{\footnotesize\rm\thepage\hfill\runningheadeven\hfill}
 \def\@oddfoot{}
 \def\@evenfoot{\@oddfoot}
}

\begin{abstract}
In \cite{De} Dempwolff gives a construction of three classes of rank
two semifields of order $q^{2n}$, with $q$ and $n$ odd, using
Dembowski-Ostrom polynomials. The question whether these semifields
are new, i.e. not isotopic to previous constructions, is left as an
open problem. In this paper we solve this problem for $n>3$, in
particular we prove that two of these classes, labeled $\cD_{A}$ and
$\cD_{AB}$, are new for $n>3$, whereas presemifields in family
${\cD}_{B}$ are isotopic to Generalized Twisted Fields for each
$n\geq 3$.
\end{abstract}

\bigskip

\par\noindent

\section{Introduction}
In \cite{De}, Ulrich Dempwolff introduced three families of rank two
presemifields of order $q^{2n}$, with $q$ and $n\geq 3$ odd,
exploiting pairs of Dembowski--Ostrom polynomials with assigned
properties. We will refer to these presemifields as ${\cD}_{A},$
${\cD}_{B}$ and ${\cD}_{AB}$, where the notation indicates the use
of the functions $A$ and $B$ in the construction from \cite{De} (see
Section \ref{sec:demp}). In the same article, the author computed
their nuclei showing that presemifields ${\cD}_{AB}$ are not
isotopic to Generalized Twisted Fields (\cite[Prop. 4.4]{De}).
However, the information about the size of the nuclei turned out to
be insufficient to address the isotopism question for these
(pre)semifields and all those previously known, which, in fact, is
left open. In this regard, in Remark $(c)$ of \cite{De}, the author
states that in view of the rapidly growing number of semifields in
recent years, the isotopism problem for the relevant (pre)semifields
appears to be a quite unpleasant task, hoping that somebody else
could take care of it. In this paper we solve this problem for
semifields in the families ${\cD}_{A}$ and ${\cD}_{AB}$ for $n>3$,
and for semifields in the family ${\cD}_B$ for $n\geq 3$. Precisely,
we prove the following result.

\begin{theorem}\label{thm:main}
The presemifields in families ${\cD}_B$, for all $n\geq 3$, are
isotopic to Generalized Twisted Fields. The presemifields in
families ${\cD}_{A}$ and ${\cD}_{AB}$ are new for all $n>3$, i.e.
they are not isotopic to any previously known semifield. Moreover,
they are not isotopic to any semifield derived from a known
presemifield using the Knuth orbit and translation dual operation.
\end{theorem}
\section{Linear Sets}

A pointset  $L$ of a projective space
$\Lambda=PG(r-1,q^n)=PG(V,\F_{q^n})$ ($q=p^h$, $p$ prime)  is said
to be an {\em $\F_q$--linear} set of $\Lambda$ if it is defined by
the non--zero vectors of an $\F_q$--vector subspace $U$ of $V$,
i.e., $$ L=L_U=\{\langle {\bf u}\rangle_{\F_{q^n}}: {\bf u}\in
U\setminus\{{\bf 0}\}\}.\quad\quad (\footnote{In what follows, in
spite of simplicity, we will write $\langle {\bf u}\rangle$ to
denote the $\F_{q^n}$--vector space generated by $\bf u$.})$$

If $dim _{\F_q} U=k$, we say that $L$ has {\it rank} $k$. When $k=r$
and  $\langle U \rangle_{\F_{q^n}}=V$, then  the $\F_q$--linear set
$L_U$ is a {\it subgeometry} of $\Lambda$ isomorphic to $PG(r-1,q)$.

Let  $\Omega=PG(W,\F_{q^n})$ be a subspace of $\Lambda$ and let $L_U$
be an $\F_q$-linear set of $\Lambda$. Then $\Omega\cap L_U$ is an
$\F_q$--linear set of $\Omega$ defined by the $\F_q$--vector
subspace $U\cap W$ and, if $dim_{\F_q}(W\cap U)=i$, we say that
$\Omega$ has {\it weight $i$} in $L_U$. Denoting by $x_i$ the number of
points of $L_U$ of weight $i$, where $i\in\{1,\dots,n\}$, if $L_U$ is an $\F_q$-linear set of
$\Lambda$ of rank $k>0$, then
\begin{eqnarray}
&&\hskip -.8cm |L_U|=x_1+x_2+\dots+x_n,\label{form1}\\
&&\hskip -.8cm x_1+(q+1)x_2+\dots+(q^{n-1}+\dots
+q+1)x_n=q^{k-1}+q^{k-2} +\dots+q+1,\label{form2}\\
&&\hskip -.8cm |L_U|\leq  q^{k-1}+q^{k-2}+\dots+q+1,\label{form4}\nonumber\\
&&\hskip -.8cm |L_U |\equiv 1\, (mod \, q)\label{form3}\nonumber.
\end{eqnarray}

 \noindent
 For further details on linear sets see \cite{Polverino2010}, \cite{LaVa2010}, \cite{LV} and \cite{LuMaPoTr-Sub}.

In \cite{LP}, the authors give the following characterization of
$\F_q$--linear sets. Let $\Sigma=PG(k-1,q)$ be a subgeometry of
$\Sigma^*=PG(k-1,q^n)$, let $\Gamma$ be a $(k-r-1)$--dimensional
subspace  of $\Sigma^*$ disjoint from $\Sigma$ and let
$\Lambda=PG(r-1,q^n)$ be an $(r-1)$--dimensional subspace of
$\Sigma^*$ disjoint from $\Gamma$. Denote by $L=\{\langle \Gamma,P
\rangle \cap \Lambda  \,: \, P\in \Sigma \}$ the projection  of
$\Sigma$ from $\Gamma$ onto $\Lambda.$ We call $\Gamma$ and
$\Lambda$, respectively, the $center$ and the $axis$ of the
projection. The map $p_{\Gamma,\Lambda}:P\in\Sigma\mapsto \langle \Gamma,P\rangle \cap \Lambda$
is surjective and $L=p_{\Gamma,\Lambda}(\Sigma).$

\begin{theorem}{\rm\cite[Theorems 1 and 2]{LP}}\label{Theor LP1}
If $L$ is a projection of $\Sigma=PG(k-1,q)$ onto
$\Lambda=PG(r-1,q^n),$ then $L$ is an $\F_q$--linear set of
$\Lambda$ of rank $k$ and $\langle L \rangle=\Lambda$. Conversely,
if $L$ is an $\F_q$--linear set of $\Lambda$ of rank $k$ and
$\langle L \rangle=\Lambda$, then either $L$ is a subgeometry of
$\Lambda$ or for each $(k-r-1)$--dimensional subspace $\Gamma$ of
$\Sigma^*=PG(k-1,q^t)$ disjoint from $\Lambda$ there exists a
subgeometry $\Sigma$ of $\Sigma^*$ disjoint from $\Gamma$ such
that $L=p_{\Gamma,\Lambda}(\Sigma).$
\end{theorem}

Following \cite{BL}, \cite{ML}, we call an $\F_q$--linear set $L_U$ of $\Lambda$ of rank $k$ {\em
scattered} if all of its points have weight 1, or equivalently, if
$L_U$ has maximum size $q^{k-1}+q^{k-2}+\cdots+q+1$.  In
\cite[Theorem 4.3]{BL}, the authors prove that a scattered
$\F_q$--linear set of $PG(r-1,q^n)$ has rank at most $rn/2$. A
scattered $\F_q$--linear set $L$ of $PG(r-1,q^n)$ of maximum rank
$rn/2$ is called a {\em maximum scattered} linear set. Note that, as
a consequence of the above mentioned result, if $L$ is a scattered
$\F_q$-linear set of $PG(r-1,q^n)$, then each line  of
$PG(r-1,q^n)$ has weight at most $n$ in $L$.

\begin{defin}{\rm
If $L$ is a scattered $\F_q$-linear set of $PG(r-1,q^n)$, a line
of weight $n$ in $L$ is said to be a {\it long line with respect to} $L$.}
\end{defin}

In \cite{LuMaPoTr-Sub}, generalizing results contained in
\cite{MPT}, \cite{LMPT} and  \cite{LV}, a family of maximum
scattered linear sets, called of {\it pseudoregulus type}, is
introduced. Precisely, a scattered $\F_q$--linear set $L=L_U$ of
$\Lambda=PG(2h-1,q^n)$ of rank $hn$ ($h,n \geq 2$) is of {\it pseoudoregulus type} if\\
(i) there exist $m=\frac{q^{nh}-1}{q^n-1}$ pairwise disjoint long
lines of $L$, say $s_1,s_2,\dots,s_m$;\\ (ii) there exist exactly
two $(h-1)$--dimensional subspaces $T_1$ and $T_2$ of $\Lambda$
disjoint from $L$ such that $T_j\cap s_i\neq \emptyset$ for each
$i=1,\dots,m$ and for each $j=1,2$.

The set of lines $\mathcal{P}_L = \{s_i \colon i=1,\dots,m\}$ is called
the {\it $\F_q$--pseudoregulus} (or simply {\it pseudoregulus}) of
$\Lambda$ associated with $L$ and  $T_1$ and $T_2$ are the  {\it
transversal spaces} of $\mathcal{P}_L$ (or {\it transversal spaces} of
$L$).

Here, we are interested in $\F_q$--linear sets of pseudoregulus
type of a 3--dimensional projective space $\Lambda=PG(3,q^n)$. In
such a case the associated pseudoregulus ${\cal P}_L$ consists of
$q^n+1$ long lines and its transversal spaces are two disjoint
lines of $\Lambda$.

By \cite[Sec. 2]{LMPT} and \cite[Thms 3.5 and 3.9]{LuMaPoTr-Sub},
$\F_q$--linear sets of pseudoregulus type of $\Lambda=PG(3,q^n)$
can be characterized in the following way.

\begin{theorem}[]\label{thm:algebraicpseudoregulus}
Let $T_1=PG(U_1,\F_{q^n})$ and $T_2=PG(U_2,\F_{q^n})$ be two
disjoint lines of $\Lambda=PG(V,\F_{q^n})=PG(3,q^n)$ and let
$\Phi_f$ be a strictly semilinear collineation between $T_1$ and
$T_2$ having as a companion automorphism an element $\sigma\in
Aut(\F_{q^n})$ such that $Fix(\sigma)=\F_q$. Then, for each $\rho
\in \F_{q^n}^*$, the set
$$L_{\rho,f} = \{\langle \underbar{u}+\rho f(\underbar{u}) \rangle \,:\, \underbar{u}\in U_1\setminus\{\underbar 0\}\}$$
is an $\F_q$-linear set of $\Lambda$ of pseudoregulus type whose
associated pseudoregulus is $\mathcal{P}_{L_{\rho,f}}=\{\langle P,
P^{\Phi_f}\rangle\,:\, P\in T_1\}$, with transversal lines $T_1$
and $T_2$.

Conversely, each $\F_q$--linear set of pseudoregulus type of
$\Lambda=PG(3,q^n)$ can be obtained as described above.
\end{theorem}

In \cite{LuMaPoTr-Sub},  $\F_q$-linear sets of pseudoregulus type
of the projective line $\Lambda=PG(V,\F_{q^n})=PG(1,q^n)$ ($n\geq
2$) are also introduced. Let $P_1=\langle \underbar w \rangle$
and $P_2=\langle \underbar v \rangle$ be two distinct points of
the line $\Lambda$ and let $\tau$ be an $\F_q$-automorphism of
$\F_{q^n}$ such that $Fix(\tau)=\F_q$; then for each $\rho \in
\F_{q^n}^*$ the set
$$W_{\rho,\tau}=\{\lambda \underbar w +
\rho\lambda^{\tau}\underbar v  \,\, \colon \,\, \lambda \in
\F_{q^n}\},$$
 is an $\F_q$--vector subspace of $V$
of dimension $n$ and  $L_{\rho,\tau}:=L_{W_{\rho,\tau}}$ is a
maximum scattered $\F_q$-linear set of $\Lambda$. The linear sets
$L_{\rho,\tau}$  are called of {\it pseudoregulus type} and  the
points $P_1$ and $P_2$ are their {\it transversal points}.

\begin{remark}\label{rem:examplenonpseudoregulustype}\rm
By \cite[Remark 4.7]{LuMaPoTr-Sub}, each $\F_q$--linear set of
pseudoregulus type of the line $\Lambda=PG(1,q^n)$ is projectively
equivalent to
\begin{equation}\label{form:NormFormLine}
L_{1,m}=\{\langle(x,x^{q^m})\rangle: x\in\F_{q^n}^*\},
\end{equation} for some
$m\in\{1,\dots,n-1\}$ with $gcd(m,n)=1$. Moreover, $L_{1,m}$ and
$L_{1,m'}$ are projectively equivalent if and only if either
$m'=m$ or $m'=n-m$.\\
Also, if $L$ is an $\F_q$-linear set of pseudoregulus type of
$PG(3,q^n)$, and $s$ is a long line of it, then $L\cap s$ is an
$\F_q$-linear set of pseudoregulus type of the line $s$ whose
transversal points are the intersection points between $s$ and the
transversal lines of $\mathcal{P}_{L}$.
\end{remark}

\begin{defin}{\rm
If $L$ is an $\F_q$-linear set of $PG(r-1,q^n)$, a
line $r$ with respect to $L$ is said to be of {\it pseudoregulus type} if the
linear set $L\cap r$ is of pseudoregulus type.}
\end{defin}

In what follows our purpose is to show that a maximum scattered
$\F_q$-linear set $L$ in $\Lambda=PG(3,q^n)$, $n\geq 3$, having
two lines of pseudoregulus type either is of pseudoregulus type or
$n\geq 5$ and it has no other long lines.

\begin{theorem}\label{thm:ScatGenDi}
Each maximum scattered $\F_q$--linear set of $\Lambda=PG(3,q^n)$,
$n\geq 3$, with at least two lines of pseudoregulus type is
projectively equivalent to the linear set
\begin{equation}\label{form:NormForm}
L_{s,t}=\{\langle (x,x^{q^s},y,y^{q^t})\rangle: x, y\in\F_{q^n}, (x,y)\neq
(0,0)\},
\end{equation}
where $s,t\in\{1,\dots,n-1\}$ with $gcd(s,n)=gcd(t,n)=1$.
Moreover, $L_{s,t}$ is of pseudoregulus type if and only if either
$s=t$ or $s=n-t$. On the other hand, if $s\neq t$ and $s\ne n-t$,
then $n\geq 5$, and there are no other long lines of $L_{s,t}$.
\end{theorem}
\begin{proof}
Let $r_1$ and $r_2$ be two lines of pseudoregulus type of a maximum scattered $\F_q$--linear set $L$ of $\Lambda$.
By \cite[Prop. 3.2]{LuMaPoTr-Sub} the lines $r_1$ and $r_2$  are
disjoint and hence,  by Remark
\ref{rem:examplenonpseudoregulustype}, up to the action of
$P\Gamma L(4,q^n)$, we may assume that
$$r_1\cap L=\{\langle(x,x^{q^s},0,0)\rangle: x\in\F_{q^n}^*\},$$
$$r_2\cap L=\{\langle(0,0,y,y^{q^t})\rangle: y\in\F_{q^n}^*\}$$ and hence
$$L_{s,t}:=L=\{\langle(x,x^{q^s},y,y^{q^t})\rangle: x, y\in\F_{q^n},
 (x,y)\neq
(0,0)\},$$
 where
$s,t\in\{1,\dots,n-1\}$ with $\gcd(s,n)=\gcd(t,n)=1$.
If $m$ is a long line with respect to $L$ and $m$ is disjoint from $r_2$, then
$$m=\{\langle (x,y,ax+by,cx+dy)\rangle ~:~x,y \in\F_{q^n}, (x,y)\neq (0,0)\},$$
for some $a,b,c,d \in \F_{q^n}$. The points of $m$ contained in $L$ must satisfy
$$
y=x^{q^s},~\mbox{and}~cx+dy=(ax+by)^{q^t},
$$
which results in $cx+dx^{q^s}-a^{q^t}x^{q^t}-b^{q^t}x^{q^{s+t}}=0$ (where $s+t$ should be thought of modulo $n$). This must hold for all $x\in \F_{q^n}$, since $m$ is a long line
with respect to $L$, and hence we obtain a polynomial identity since the degree is less than $q^n$.
If $s\neq t$ and $s+t\neq 0$ (modulo $n$), then $n\geq 5$, $b=d=c=a=0$ and $m=r_1$, and it follows that $r_1$ and $r_2$ are the only long lines
with respect to $L$. If $s=t$ then $d=a^{q^t}$ and $c=b=0$ (note that in this case $s+t\neq 0$ modulo $n$), and it follows that there are exactly $q^n+1$ long lines with respect to $L$.
Moreover the line $t_1$, generated by $\langle (1,0,0,0)\rangle$ and $\langle (0,0,1,0)\rangle$, and the line
$t_2$, generated by $\langle (0,1,0,0)\rangle$ and $\langle (0,0,0,1)\rangle$, both intersect each of these long lines.
We may conclude that $L$ is of pseudoregulus type.
The case $s+t=n$ is completely analogous to the case $s=t$ and is left to the reader.
\end{proof}

\section{Linear set associated with a rank 2 semifield}\rm
The semifields that we study here are 2n-dimensional $\F_q$-algebras
having at least one nucleus of order $q^n$: {\it rank two semifields
of order $q^{2n}$}. The Knuth orbit of such a semifield contains an
isotopism class $[\bS]$ whose left nucleus has size ${q^n}$, and
with the semifield $\bS$ there is associated an $\F_q$-linear set
$L(\bS)$ of rank $2n$, disjoint from the hyperbolic quadric
$\cQ=Q^+(3,q^n)$ of the projective space $PG(3,q^n)$ with equation
$X_0X_3-X_1X_2=0$. The isotopy class $[\bS]$ corresponds to the
orbit of $L(\bS)$ under the subgroup $\cG \leq {\mathrm{P\Gamma
O}}^+(4,q^n)$ fixing the reguli of $\cQ$. Furthermore, the image of
$L(\bS)$ under an element of ${\mathrm{P\Gamma O}}^+(4,q^n)\setminus
\cG$, defines a (pre)semifield which is isotopic to the {\it
transpose} semifield $\bS^t$ of $\bS$ (for more details see
\cite{LaPo2011}).

Let $Tr_{q^n/q}$ denote the trace function of $\F_{q^n}$ over $\F_q$
and let $b(X,Y)$ be the bilinear form associated with $\cQ$. By
field reduction we can use the bilinear form $Tr_{q^n/q}(b(X,Y))$ to
obtain another $\F_q$-linear set, say $L(\bS)^{\perp}$, of rank $2n$
disjoint from $\cQ$. The linear set $L(\bS)^\perp$ defines a rank
two semifield, say $\bS^{\perp}$, as well; such a semifield is
called the {\it translation dual} of $\bS$, see e.g.
\cite{LuMaPoTr2008}, or \cite[\S 3]{LaPo2011}. We will refer to
$\bS^t$ and $\bS^{\perp}$ as the {\it rank two derivatives} of
$\bS$.

\bigskip

\noindent Here we list the known families of (pre)semifields
2-dimensional over their left nucleus, $2n$ dimensional over their
center $\F_q$ and exiting for infinite values of $n\geq 2$, pointing out
their nuclei and describing the geometric structure of the
associated linear set with respect to the action of the group $\cG$.
For a complete list of the known (pre)semifields see \cite[\S
6]{LaPo2011}.

\begin{itemize}
\item[${\cal K}_{17}$]  {\em Knuth semifields} :\quad $(\F_{q^n}\times\F_{q^n},+,\star)$,
$$ (u,v) \star (x,y)=(xu+y^{q^{r}}vf, uy+vx^{q^r}+vy^{q^{r}}g)=(u,v)\left(
                                                                      \begin{array}{cc}
                                                                        x & y \\
                                                                        fy^{q^r} & x^{q^r}+gy^{q^r} \\
                                                                      \end{array}
                                                                    \right), $$
where $gcd(r,n)=1$, $f,g\in \F_{q^n}$ with $x^{q^r+1}+gx-f\neq 0$ $\,
\forall x\in  \F_{q^n}$. These semifields also have the middle
nucleus of order $q^n$.
The associated linear set is
$$L_{{\cal K}_{17}}=\{\langle(x,y,fy^{q^r},x^{q^r}+gy^{q^r})\rangle\,:\, x,y \in
\F_{q^n}, (x,y)\ne (0,0)\},$$ and it is a scattered $\F_q$--linear set of pseudoregulus type
with the transversals both contained in the quadric $\cQ$ (see
\cite[Prop. 5.7]{LuMaPoTr-Sub}).

\item[${\cal K}_{19}$]  {\em Knuth semifields} :\quad $(\F_{q^n}\times\F_{q^n},+,\star)$,
$$ (u,v) \star (x,y)=(xu+y^{q^{n-r}}vf, uy+vx^{q^r}+vyg)=(u,v)\left(
                                                                      \begin{array}{cc}
                                                                        x & y \\
                                                                        fy^{q^{n-r}} & x^{q^r}+gy\\
                                                                      \end{array}
                                                                    \right),$$ where
$gcd(r,n)=1$, $f,g\in \F_{q^n}$ with $x^{q^r+1}+gx-f\neq 0$ $\, \forall
x\in  \F_{q^n}$. These presemifields have also the right nucleus of
order $q^n$. The associated linear set is
$$L_{{\cal K}_{19}}=\{\langle(x,y,fy^{q^{n-r}},x^{q^r}+gy)\rangle\,:\, x,y \in \F_{q^n},  (x,y)\ne (0,0)\},$$
and, also in this case, it is a scattered $\F_q$--linear set of pseudoregulus
type with the transversals both contained in the quadric $\cQ$ (see
\cite[Prop. 5.7]{LuMaPoTr-Sub}).

\item[$\mathcal{TP}$] {\em Thas--Payne symplectic semifields and their translation duals} {\rm\cite{TP1994}}: \quad $(\F_{3^{2n}},+,\star)$\, $n \geq 2$.
The linear set associated with any $\mathcal TP$ symplectic
semifield is a union of lines contained in a plane of $PG(3,3^n)$;
whereas, the linear set associated with its translation dual
$\mathcal{TP}^{\perp}$  is a union of lines passing through a point
(see \cite{Lunardon2003}).

\item[$\mathcal{GD}$] {\em Generalized Dickson semifields} {\rm\cite{Dickson1906}}:\quad $(\F_{q^n}\times\F_{q^n},+,\star)$\, $n>1$,
$$ (u,v) \star (x,y)=(xu+vy^{q^t}f,uy+vx^{q^s})=(u,v)\left(
                                                                      \begin{array}{cc}
                                                                        x & y \\
                                                                        fy^{q^t} & x^{q^s}\\
                                                                      \end{array}
                                                                    \right),$$
where $s,t\in\{0,1,\dots,n-1\}$, $(s,t)\ne (0,0)$, $gcd(s,t,n)=1$
and $f\in \F_{q^n}$ such that $x^{q^s+1}-fy^{q^t+1}\neq 0$ $\,
\forall\, (x,y) \in  \F_{q^n}^2\setminus\{(0,0)\}$ (see
{\rm\cite{Dickson1905}}, {\rm\cite{Dickson1906}},
{\rm\cite{Dickson1906-1}}). These presemifields have middle nucleus
of order $q^{gcd(t-s,n)}$ and right nucleus of order
$q^{gcd(t+s,n)}$. The associated linear set is
$$L_{\cGD}=\{\langle(x,y,fy^{q^t},x^{q^s})\rangle\,:\, x,y \in
\F_{q^n}, (x,y)\ne (0,0)\}.$$

If $s=0$, then the associated linear set $L_{\cGD}$ is a union of lines through the point $(1,0,0,1)$ of weight $n$ and it is contained in the plane $x_0=x_3$. When $t=0$, we have the same geometric configuration.

If $s,t>0$, then  $L_{\cGD}$ is not contained in a plane and the
following cases occur:\medskip

\fbox{Scattered case.}\quad In this case  $gcd(s,n)=gcd(t,n)=1$ and $L_{\cGD}$ is equivalent (up to the action of the group $P\Gamma L(4,q^n)$) to the linear set (\ref{form:NormForm}) of Theorem \ref{thm:ScatGenDi}. Hence, it is a maximum
scattered $\F_q$-linear set of $PG(3,q^n)$ and by Theorem
\ref{thm:ScatGenDi}, $L_{\cGD}$ is of pseudoregulus type if and only
if either $s=t$ or $s+t=n$. Also, in such a case $\cGD$ is a
2-dimensional Knuth semifield $\cK_{17}$ or $\cK_{19}$ (with $g=0$). In the other cases, i.e. $n\geq 5$,
$s\neq t$ and $s\neq n-t$, by Theorem \ref{thm:ScatGenDi}, the lines $r:x_1=x_2=0$ and $r^\perp:x_0=x_3=0$ (where $\perp$ is the polarity induced by the quadric $\cQ=Q^+(3,q^n)$) are
the only long lines of the linear set $L_{\cGD}$ of $PG(3,q^n)$
and by Remark \ref{rem:examplenonpseudoregulustype}, they are both of pseudoregulus
type.

\fbox{Non--scattered case.} We can have the following possibilities:\\
$\bullet$\ \ $gcd(s,n)=h>1$ and $gcd(t,n)=1$. \quad  In this case the points of weight greater than 1 are only on the line $r:x_1=x_2=0$ and each of them has weight $h$. Hence, $|L_{\cGD}|=\frac{q^{2n}-q^n}{q-1}+\frac{q^n-1}{q^h-1}$. \\
$\bullet$\ \ $gcd(t,n)=k>1$ and $gcd(s,n)=1$. \quad  In this case the points of weight greater than 1 are only on the line $r^\perp:x_0=x_3=0$ and each of them has weight $k$. Hence, $|L_{\cGD}|=\frac{q^{2n}-q^n}{q-1}+\frac{q^n-1}{q^k-1}$.\\
$\bullet$\ \ $gcd(s,n)=h>1$ and $gcd(t,n)=k>1$. \quad  In this case the points of weight greater than 1 are only on the lines $r$ and $r^\perp$ and each of them has weight $h$ and $k$, respectively. Hence, $|L_{\cGD}|=\frac{(q^{n}-1)^2}{q-1}+\frac{q^n-1}{q^h-1}+\frac{q^n-1}{q^k-1}$.

\item[$\cA$] {\em Generalized Twisted Fields of rank 2} {\rm\cite{Albert1961P}}:\quad $(\F_{q^{2n}},+,\star)$,  with
$ x \star y=yx-cy^{q^t}x^{q^n},$ where $gcd(t,n)=1$  and $c \neq y^{1-q^t}x^{1-q^n}$ for each $x,y\in\F_{q^{2n}}^*$.
These presemifields have middle nucleus of size $q^{gcd(t+n,2)}$ and
right nucleus of size $q^{gcd(t,2)}$ {\rm (see \cite{Albert1961A})}.
The associated linear set is an $\F_q$--linear set of $PG(3,q^n)$ of
pseudoregulus type with transversal lines external to the quadric
$\cQ$ and pairwise polar with respect to the polarity defined by
$\cQ$ (see \cite[Corollary 5.5]{LuMaPoTr-Sub}).

\item[${\mathcal {JMPT}}$] {\em Johnson--Marino--Polverino--Trombetti semifields} {\rm\cite{JoMaPoTr2009}}:\quad $(\F_{q^{2n}},+,\star)$,  with $n\geq 3$ odd. These semifields have middle and right nuclei both of order $q^2$ (\cite[Thm. 1]{JoMaPoTr2009}). The associated linear set $L$ contains a unique line of $PG(3,q^n)$ and there are at least $q+1$ points of this line with weight $\frac{n+1}2$ in $L$ (\cite[Thm. 2]{JoMaPoTr2009} and also \cite[Prop. 3.2, Remark 3.4]{EbMaPoTr2009}).

\item[${\mathcal {EMPT}1}$] {\em Ebert--Marino--Polverino--Trombetti semifields} (odd case) {\rm\cite{EbMaPoTr2009}}:\quad $(\F_{q^{2n}},+,\star)$,  with $n\geq 3$ odd. These semifields have middle and right nuclei both of order $q$ (\cite[Thms 4.1,4.2,4.3,4.4,4.5]{EbMaPoTr2009}). The associated linear set $L$ contains a unique line of $PG(3,q^n)$ and there are $q+1$ points of this line with weight $\frac{n+1}2$ in $L$ (see \cite[Prop. 3.2]{EbMaPoTr2009}).

\item[${\mathcal {EMPT}2}$] {\em Ebert--Marino--Polverino--Trombetti semifields} (even case) {\rm\cite{EbMaPoTr2009}}:\quad $(\F_{q^{2n}},+,\star)$,  with $n\geq 4$ even. These semifields have middle and right nuclei both of order $q^2$ (\cite[Thm. 4.6]{EbMaPoTr2009}). The geometric structure of the associated linear is described in \cite[Prop. 3.3]{EbMaPoTr2009}.
\end{itemize}

We end this section observing that the transposes of the above
listed presemifields remain in the list. Indeed, we have that the
transpose of  a semifield of type ${\cal K}_{17}$ is a semifield of
type ${\cal K}_{19}$ (\cite[\S 5]{Knuth1965}), whereas the other
families are closed under taken the transpose (\cite[\S
4]{LuMaPoTr2008}, \cite[Cor. 3]{JoMaPoTr2009}, \cite{EbMaPoTr2009}).

Finally, we determine the linear set associated with the translation
dual of a presemifield defined by a linear set of type
$$L_{f,g}=\{\langle(x,y,f(y), g(x))\rangle\, : \, x,y \in \F_{q^n}, (x,y) \neq (0,0)\} \subset PG(3,q^n),$$
where $f$ and $g$ are $\F_q$--linear maps of $\F_{q^n}$. In order to
do this, if $a(x)=\sum_{i=0}^{n-1}a_i x^{q^i}$, we will denote by
$\hat{a}(x)=\sum_{i=0}^{n-1}a_i^{q^{n-i}} x^{q^{n-i}}$ the adjoint
polynomial of $a(x)$ with respect to the non-degenerate
$\F_q$-bilinear form $\langle x,y \rangle = Tr_{q^n/q}(xy)$ of
$\F_{q^n}$, i.e. $Tr_{q^n/q}(x\,a(y))=Tr_{q^n/q}(\hat{a}(x)\,y)$. We
have the following

\begin{prop}\label{prop:translation_dual}
The isotopism class of the translation dual of a semifield defined
by the linear set
$$L_{f,g}=\{\langle(x,y,f(y), g(x))\rangle\, : \, x,y \in \F_{q^n}, (x,y) \neq (0,0)\} \subset PG(3,q^n),$$
is represented by the presemifield corresponding to the linear set
$$L_{f,g}^{\perp}=\{\langle (x,y, \hat{f}(y), \hat{g}(x))\rangle\, : \, x,y \in \F_{q^n}, (x,y) \neq (0,0)\} \subset
PG(3,q^n),$$ where $\perp$ is the translation dual operation with respect to the bilinear form defined by the hyperbolic quadric with equation $X_0X_3-X_1X_2=0$.
\end{prop}
\begin{proof}
Let $\cal{Q}$ be the hyperbolic quadric of $PG(3,q^n)$ with equation
$X_0X_3-X_1X_2=0$ and $b({\bf X},{\bf
Y})=X_3Y_0-X_2Y_1-X_1Y_2+X_0Y_3$  be the bilinear form associated
with $\cQ$. One easily verifies (using the same calculation as in
\cite[page 3-4]{Lavrauw2005c} to obtain the subspace $U^D$ from $U$)
that the linear set obtained from
$$L_{f,g}=\{\langle(x,y,f(y), g(x))\rangle\, : \, x,y \in \F_{q^n}, (x,y) \neq (0,0)\} \subset PG(3,q^n),$$
performing the translation dual operation is
$$L_{f,g}^{\perp}=\{\langle (z, w,-\hat{f}(w),-\hat{g}(z))\rangle\, : \, w,z \in \F_{q^n}, (w,z) \neq (0,0)\} \subset PG(3,q^n);
$$ indeed $$Tr_{q^n/q}(b((x,y,f(y),g(x)),(z,w,-\hat{f}(w),-\hat{g}(z))))=Tr_{q^n/q}(g(x)z-f(y)w+y\hat{f}(w)-x\hat{g}(z))=0,$$
and  $L_{f,g}^{\perp}$  is equivalent under $\cG$ to
$$L_{f,g}^{\perp}=\{\langle (x,y, \hat{f}(y), \hat{g}(x))\rangle\, : \, x,y \in \F_{q^n}, (x,y) \neq (0,0)\} \subset PG(3,q^n).
$$
\end{proof}

\section{Dempwolff's presemifields}\label{sec:demp}

In \cite{De} the author introduced three families of rank 2 presemifields,
computing their nuclei and leaving open the question to investigate
the isotopy issue between these presemifields and each presemifield
previously known. These presemifields are obtained starting from a
pair $(F_1,F_2)$ of bijective $\F_q$-linear maps of $\F_{q^n}$, {\bf
$q$ odd and $n$ odd}, such that
\begin{equation} \label{Comd:Demp}
|P_{F_1}(\F_{q^n}^*)|=|P_{F_2}(\F_{q^n}^*)|=\frac{q^n-1}{2} \quad
\mbox{and} \quad P_{F_1}(\F_{q^n}^*) \cap \xi P_{F_2}(\F_{q^n}^*) =
\emptyset,
\end{equation} where $P_{F_i}(x)=F_i(x)x$ and $\xi$ is a fixed non-square
element of $\F_q$. Precisely, the algebraic structure
$\bS(F_1,F_2)=(\F_{q^n}\times \F_{q^n},+,\ast),$ whose
multiplication is given by
$$(u,v)\ast(x,y)=(u,v)\begin{pmatrix} x & y \\ F_1(y) & \xi F_2(x)
\end{pmatrix}, \quad \quad (*)$$ \noindent where $F_1$ and $F_2$ are bijective $\F_q$-linear maps of $\F_{q^n}$ satisfying Conditions (\ref{Comd:Demp}), is a presemifield with dimension at most
2 over the left nucleus (see \cite[Theorem 4.1]{De}). The three
families of proper presemifields $\bS(F_1,F_2)$ constructed by
Dempwolff are obtained considering $F_1$ and $F_2$ as follows:
\begin{itemize}
\item[$A)$]  $F_1(x)=F_2(x)=A_{a,r}(x)=x^{q^r}-ax^{q^{-r}}$ such
that $gcd(n,r)=1$ and $a$ is an element of $\F_{q^n}^*$ with $N_q(a) \neq 1$, where $N_q\,:\,\F_{q^n} \rightarrow \F_{q}$ denotes the norm function from $\F_{q^n}$ over $\F_q$; \item[$B)$]
$F_1(x)=F_2(x)=B_{b,r}(x)=2H_{b,r}^{-1}(x)-x$ such that
$H_{b,r}(x)=x-bx^{q^r}$, $gcd(n,r)=1$ and $b \in \F_{q^n}^*$ with $N_q(b) \neq \pm 1$;
\item[$AB)$]  $F_i(x)=A_{b^2,r}(x)$, $F_j(x)=B_{b,-r}(x)$,
$\{i,j\}=\{1,2\}$ such that $gcd(n,r)=1$ and $N_q(b) \neq \pm 1$.
\end{itemize}
We will denote the corresponding presemifields as $\cD_A$, $\cD_B$
and $\cD_{AB}$, respectively. It is easy to see that
$\bS(F_1,F_2)$ is isotopic to $\bS(F_2,F_1)$, so in what follows
 we assume that, in case $AB)$,
$F_1(x)=A_{b^2,r}(x)$ and  $F_2(x)=B_{b,-r}(x)$. Moreover,
$\bS(F_1^{-1},F_2)$ is obtained from $\bS(F_1,F_2)$ by transposition
as mentioned in \cite[\S 4]{De}.

\begin{remark}{\rm (\cite[Theorem 4.3]{De})\label{rem:nucleiDemp}
All presemifields $\cD_A$, $\cD_B$ and $\cD_{AB}$ have left nucleus of size $q^n$ and middle nucleus and center both of size $q$. Moreover, presemifields $\cD_A$ and $\cD_B$ have right nucleus of size $q^2$, whereas presemifields $\cD_{AB}$ have right nucleus of size $q$.}
\end{remark}

In order to compare Dempwolff presemifields with the known
(pre)semifields listed in the previous section, in the next,  we
will investigate the geometric structure of the associated linear
sets.

Let $\bS(F_1,F_2)$  be a presemifield with multiplication $(*)$,
then the associated $\F_q$-linear set is the following
$$L(F_1,F_2)=\{\langle(x,y,F_1(y),\xi F_{2}(x))\rangle\, : \, x,y \in \F_{q^n}, (x,y) \neq (0,0)\} \subset PG(3,q^n).$$

\noindent The linear set $L(F_1,F_2)$ admits two long lines $r_1$ and $r_2$, i.e.
lines of weight $n$ in $L(F_1,F_2)$, both  secants to the quadric
$\cQ=Q^+(3,q^n):\, X_0X_3-X_1X_2=0$,  which are pairwise polar with
respect to the polarity $\perp$ defined by  $\cQ$; in fact, $$
r_1:X_1=X_2=0 \quad \text{and}\quad  r_2=r_1^{\perp}: X_0=X_3=0.$$

Concerning the translation dual we
have the following.
\begin{prop}\label{prop:transl-dual}
The translation dual of a presemifield of type $\cD_A$ ($\cD_B$, $\cD_{AB}$, respectively) is a presemifield is of type $\cD_A$ ($\cD_B$, $\cD_{AB}$, respectively).
\end{prop}
\begin{proof}
Applying Proposition \ref{prop:translation_dual}, since $\hat A_{a,r}(x)=A_{a^{q^r},-r}(x)$ and $\hat B_{b,r}(x)=B_{b^{q^{-r}},-r}(x)$ we have that the translation dual of a presemifield of type $\cD_A$ ($\cD_B$, $\cD_{AB}$, respectively) is a presemifield is of type $\cD_A$ ($\cD_B$, $\cD_{AB}$, respectively) as well.
\end{proof}

\subsection*{$\cD_A$ presemifields}

The linear set associated with a presemifield $\cD_A$ is
$$L_{\cD_A}=\{\langle(x,y,A_{a,r}(y),\xi A_{a,r}(x))\rangle\, : \, x,y \in \F_{q^n}, (x,y) \neq (0,0)\} \subset PG(3,q^n),$$
where  $A_{a,r}(x)=x^{q^r}-ax^{q^{-r}}$ with $gcd(n,r)=1$ and $a\in\F_{q^n}^*$ such that $N_q(a)\ne 1$, and $\xi$ a fixed
nonsquare in $\F_q$.\\

We start with the following lemma.
\begin{lemma} \label{lemma:LineCASEA}
Let $L_A$ be the following $\F_q$-linear set of rank $n$ of
the line $PG(1,q^n)$\\
$$L_A=\{\langle(x,A_{a,r}(x))\rangle\,:\, x\in\F_{q^n}^*\},
$$ where  $A_{a,r}(x)=x^{q^r}-ax^{q^{-r}}$ such that $gcd(n,r)=1$ and $N_q(a)\ne 1$. Then the following statements hold true:
\begin{itemize}
\item[$(i)$] $L_A$ is scattered if and only if $N_q(a) \neq -1$ and in such a
case, if $n>3$,  $L_A$ is not of pseudoregulus type.
\item[$(ii)$]  If $N_q(a) = -1$, then the points of $L_A$ of weight greater
than one have weight $2$, and they are $(q^{n-1}-1)/(q^2-1)$ in number,
hence $|L_A|=q^{n-1}+\frac{q^{n-1}-1}{q^2-1}$.
\end{itemize}
\end{lemma}
\begin{proof}
$(i)$ Let $P_{ x}=\langle( x,A_{a,r}( x))\rangle$ be a point of $L_A$ and suppose that its weight is greater than 1. Then there exists $\lambda\in\F_{q^n}\setminus\F_q$ such that
$\lambda x= x'$ and $\lambda A_{a,r}( x)=A_{a,r}( x')$, for some $x'\in\F_{q^n}^*$. These conditions imply that $\lambda A_{a,r}( x)=A_{a,r}(\lambda x)$ and this is equivalent to the following equation:
\begin{equation}\label{formA-1}
a^{q^r}(\lambda^{q^r}-\lambda) x+(\lambda^{q^{2r}}-\lambda^{q^{r}}) x^{q^{2r}}=0.
\end{equation}
The previous equation is satisfied if and only if $ x$ is a solution of
\begin{equation}\label{form:star}
x^{q^{2r}-1}=-\frac {a^{q^r}}{(\lambda-\lambda^{q^r})^{q^r-1}},
\end{equation}
and, since $gcd(n,2r)=1$, this happens if and only if $N_q(a)=-1$.
This means that $L_A$ is a maximum scattered $\F_q$--linear set of
$PG(1,q^n)$ if and only if $N_q(a)\ne -1$.

Suppose that $N_q(a) \ne -1$. By Remark \ref{rem:examplenonpseudoregulustype}, $L_A$ is of pseudoregulus type if and only if $L_A$ is projectively equivalent to an $\F_q$--linear set of type (\ref{form:NormFormLine}), i.e. if and only if there exit $\alpha,\beta,\gamma,\delta\in\F_{q^n}$, with $\alpha\delta-\beta\gamma\ne 0$, such that
\begin{equation}
\left(
  \begin{array}{cc}
    \alpha & \beta \\
    \gamma & \delta \\
  \end{array}
\right)\left(
         \begin{array}{c}
           z \\
           z^{q^m} \\
         \end{array}
       \right)=\left(
                 \begin{array}{c}
                   x \\
                   A_{a,r}(x) \\
                 \end{array}
               \right)
\end{equation}
for each $x,z\in\F_{q^n}$ and for some $m\in\{1,\dots,n-1\}$ with $\gcd(m,n)=1$. The last equation implies that
\begin{equation}\label{formA-3}
\gamma z+\delta z^{q^m}=\alpha^{q^r}z^{q^r}+\beta^{q^r}z^{q^{m+r}}-a\alpha^{q^{-r}}z^{q^{-r}}-a\beta^{q^{-r}}z^{q^{m-r}},
\end{equation} for each $z\in\F_{q^n}$.
Since $n$ is odd and $gcd(n,r)=1$, then $m+r\not \equiv m-r (mod\,n)$. This means that, if $m+r\not\equiv 0 (mod\,n)$ and $m-r\not\equiv 0 (mod\,n)$, then $\beta=\gamma=\delta=0$, a contradiction. Hence, either $m+r\equiv 0 (mod\,n)$ or $m-r\equiv 0 (mod\,n)$.

If
\begin{equation}\label{formA-4}
m+r\equiv 0 (mod\,n),
\end{equation}
since $\alpha \gamma - \beta \delta \ne 0,$ from (\ref{formA-3}) we get
\begin{equation}\label{formA-5}
m-r\equiv r (mod\,n),
\end{equation}
otherwise $\alpha=\beta=0$. Since $gcd(n,r)=1$, Equations
(\ref{formA-4}) and (\ref{formA-5}) are both satisfied if and only
if $n=3$ and $\{m,r\}=\{1,2\}$.

If $m-r\equiv 0 (mod\,n)$, arguing as above we get that $n=3$ and $m=r\in\{1,2\}$. Hence, Statement $(i)$ is proved.

\bigskip

$(ii)$ Suppose, now, that $N_q(a)=-1$ and let $P_{x_0}=\langle (x_0,A_{a,r}(x_0))\rangle$, $x_0\in\F_{q^n}^*$, be a point of $L_A$ with weight greater than 1. Then there exists $\lambda_0\in\F_{q^n}\setminus\F_q$ such that
\begin{equation}\label{formA-6}
A_{a,r}(\lambda_0 x_0)=\lambda_0 A_{a,r}(x_0),
\end{equation} i.e., the pair $(x_0,\lambda_0)$ satisfies Equation (\ref{form:star}).
This means that
$$\{(\alpha x_0+\beta\lambda_0x_0,A_{a,r}(\alpha x_0+\beta\lambda_0x_0)): \alpha,\beta\in\F_q\}\subseteq \langle x_0\rangle\cap \{(x,A_{a,r}(x))\,:\, x\in\F_{q^n}\},$$
and hence $P_{x_0}$ has weight at least 2 in $L_A$.

If $P_{x_0}$ had weight greater than 2, then there would be
$\mu\in\F_{q^n}$ and $\mu \notin \langle 1,\lambda_0 \rangle_{\F_q}$
such that
\begin{equation}\label{formA-7}
A_{a,r}(\mu x_0)=\mu A_{a,r}(x_0).
\end{equation}

From Equation (\ref{form:star}), we get
$\frac{\mu^{q^r}-\mu}{\lambda_0^{q^r}-\lambda_0}\in\F_q^*$, and this
happens if and only if $\mu\in\langle 1,\lambda_0\rangle_{\F_q}$, a
contradiction. Hence
$$\langle x_0\rangle\cap \{(x,A_{a,r}(x))\,:\, x\in\F_{q^n}\}=\{(\alpha x_0+\beta\lambda_0x_0,A_{a,r}(x_0+\alpha\lambda_0x_0)):\alpha,\beta\in\F_q\},$$ i.e. $P_{x_0}$ has weight 2 in $L_A$.

It follows that the number of pairs $(\lambda ,x)$, with $\lambda
\in {\mathbb{F}}_{q^n}\setminus{\mathbb{F}}_q$, corresponding to a
point $\langle (x, A_{a,r}(x))\rangle$ of weight 2 is
$(q^n-q)(q-1)$, since for a given $\lambda  \in
{\mathbb{F}}_{q^n}\setminus{\mathbb{F}}_q$, Equation
(\ref{form:star}) admits $q-1$ solutions $x \in {\mathbb{F}}_{q^n}$.
Since two equations $A_{a,r}(\lambda x)=\lambda A_{a,r}(x)$ and
$A_{a,r}(\mu x)=\mu A_{a,r}(x)$ have the same solution $x \in
{\mathbb{F}}_{q^n}$ if and only if $\mu \in \langle 1, \lambda
\rangle_{{\mathbb{F}}_q}$, and in such a case they admits the same
set of solutions, we have exactly
$$
\frac{(q^n-q)(q-1)}{q^2-q}=q^{n-1}-1
$$
values $x\in \mathbb{F}_{q^n}$ that define a point $\langle (x,
A_{a,r}(x))\rangle$ of weight 2, and hence
$$
q^n-q^{n-1}
$$
values  $x\in \mathbb{F}_{q^n}$ defining a point of weight 1. So,
there are
$$
\frac{(q^{n-1}-1)}{q^2-1}
$$
points of weight 2 and $\frac{q^n-q^{n-1}}{q-1}=q^{n-1}$ points of
weight 1. Since $L_A$ has no points of weight greater than 2, we
have
$$
|L_A|=q^{n-1}+\frac{(q^{n-1}-1)}{q^2-1}.
$$
\end{proof}

\begin{theorem}  \label{thm:linearsetCaseA}
The $\F_q$-linear set $L_{\cD_A}$ has the following structure:
\begin{itemize}
\item [$i)$] If $N_q(a) \neq -1$ and $n=3$, then $L_{\cD_A}$ is a maximum
scattered $\F_q$-linear set of $PG(3,q^3)$ of pseudoregulus type.
\item [$ii)$] If $N_q(a) \neq -1$ and $n>3$, then $L_{\cD_A}$ is a maximum
scattered $\F_q$-linear set of $PG(3,q^n)$ with two long lines not
of pseudoregulus type.
\item [$iii)$] If $N_q(a)= -1$, then the points of $L_{\cD_A}$ of weight bigger
than 1 have weight $2$ and they are $\frac{q^{n-1}-1}{q-1}$ in
number. Hence, $|L_{\cD_A}|=\frac{q^{2n}-q^n}{q-1}+1$.
\end{itemize}
\end{theorem}

\begin{proof}
We explicitly note that $L_{\cD_A}$ is contained in the lines
joining a point of $L_{\cD_A}\cap r_1$ with a point of
$L_{\cD_A}\cap r_1^\perp$, where
$$L_{\cD_A}\cap r_1=\{\langle(x,0,0,\xi A_{a,r}(x))\rangle\,:\,
x\in\F_{q^n}^*\}$$ and $$L_{\cD_A}\cap
r_1^\perp=\{\langle(0,y,A_{a,r}(y),0)\rangle\,:\, y\in\F_{q^n}^*\}$$
and both the linear sets $L_{\cD_A}\cap r_1$ and $L_{\cD_A}\cap
r_1^\perp$ are of type described in Lemma \ref{lemma:LineCASEA}.

Let $P_{x,y}=\langle( x, y,A_{a,r}( y),\xi A_{a,r}( x))\rangle$ be a
point of $L_{\cD_A}\setminus (r_1\cup r_1^\perp)$ and suppose that
its weight is greater than 1. Then there exists
$\lambda\in\F_{q^n}\setminus \F_q$ such that $\lambda A_{a,r}(
x)=A_{a,r}(\lambda x)$ and $\lambda A_{a,r}( y)=A_{a,r}(\lambda y)$.
Arguing as in the proof of Lemma \ref{lemma:LineCASEA}, this happens
if and only if $N_q(a)=-1$ and, in such a case, $ x$ and $ y$ are
solutions of the same equation $A_{a,r}(\lambda x)=\lambda
A_{a,r}(x)$, hence $ y=\rho  x$, where $\rho\in\F_q^*$, and $P_{ x,
y}$ has weight 2 in $L_{\cD_A}$. This means that, if $N_q(a)\ne -1$,
then $L_{\cD_A}$ is a maximum scattered $\F_q$--linear set of
$PG(3,q^n)$ and by \cite[Prop. 2.8]{MPT} and by part $(i)$ of Lemma
\ref{lemma:LineCASEA}, the assertions $i)$ and $ii)$ follow. On the
other hand, if $N_q(a)=-1$ and $P\notin r\cup r^\perp$ is a point of
weight 2 in $L_{\cD_A}$, then $P$ belongs to the line
$\ell_{x_0}:=\langle P_{x_0,0},P_{0,x_0}\rangle$, where
$x_0\in\F_{q^n}^*$ such that $A_{a,r}(\lambda x_0)=\lambda
A_{a,r}(x_0)$ for some $\lambda\in\F_{q^n}\setminus\F_q$ and
$P=P_{x_0,\rho x_0}$ with $\rho\in\F_q^*$. Then there are $q+1$
points on the line $\ell_{x_0}$ of weight 2 in $L_{\cD_A}$
($P_{x_0,0}$ and $P_{0,x_0}$ included). From $ii)$ of Lemma
\ref{lemma:LineCASEA} there are $(q+1)\frac{q^{n-1}-1}{q^2-1}$
points of weight 2 in $L_{\cD_A}$. Using Equations (\ref{form1}) and
(\ref{form2}), the last part of statement $iii)$ follows.
\end{proof}

\subsection*{$\cD_B$ presemifields}

In this Section we will describe the linear set associated with
presemifields $\cD_B$, precisely:
$$L_{\cD_B}=\{\langle(x,y,B_{b,r}(x),\xi B_{b,r}(x))\rangle\, : \, x,y \in \F_{q^n}, (x,y) \neq (0,0)\} \subset PG(3,q^n),$$
where $B_{b,r}(x)=2H_{b,r}^{-1}(x)-x$ with $H_{b,r}(x)=x-bx^{q^r}$, $gcd(n,r)=1$ and $b\in\F_{q^n}^*$ such that $N_q(b)\ne \pm 1$.\\

Start with the following lemma.
\begin{lemma} \label{lemma:LineCASEB}
Let $L_B$ be the following $\F_q$-linear set of rank $n$ of
the line $PG(1,q^n)$\\
$$L_B=\{\langle(x,B_{b,r}(x))\rangle\,:\, x\in\F_{q^n}^*\}
$$
where  $B_{b,r}(x)=2H_{b,r}^{-1}(x)-x$ such that
$H_{b,r}(x)=x-bx^{q^r}$, $gcd(n,r)=1$ and $N_q(b)\ne \pm 1$. Then
 $L_B$ is a maximum scattered $\F_q$--linear set of pseudoregulus type.\\
\end{lemma}
\begin{proof}
Since $H_{b,r}(x)$ defines an invertible $\F_q$--linear map of
$\F_{q^n}$, the linear set $L_B$ can be rewritten as
$$L_B=\{\langle(H_{b,r}(x),H_{b,r}(B_{b,r}(x)))\rangle\,:\,
x\in\F_{q^n}^*\}=\{\langle(x-bx^{q^r},x+bx^{q^r})\rangle \,:\,
x\in\F_{q^n}^*\}.$$ Since $q$ is odd, arguing as in the proof of
Lemma \ref{lemma:LineCASEA}, straightforward computations show that
$L_B$ is a maximum scattered $\F_q$-linear set of $PG(1,q^n)$ and it
is projectively equivalent to the linear set $\{\langle (x,
x^{q^r})\rangle: x\in\F_{q^n}^*\}$. By Remark
\ref{rem:examplenonpseudoregulustype}, the assertion follows.
\end{proof}

\begin{theorem}  \label{thm:linearsetCaseB}
The $\F_q$-linear set $L_{\cD_B}$ is a maximum scattered
$\F_q$-linear set of pseudoregulus type whose transversal lines are
external to the quadric $\cQ$ and pairwise polar w.r.t. the polarity
defined by $\cQ$.
\end{theorem}
\begin{proof}
As in the proof of Lemma \ref{lemma:LineCASEB}, we explicitly note
that $L_{\cD_B}$ can be rewritten as
$$L_{\cD_B}=\{\langle(x-bx^{q^r},y-by^{q^r},y+by^{q^r},\xi
(x+bx^{q^r}))\rangle\, : \, x,y \in \F_{q^n}, (x,y) \neq (0,0)\}.$$
Hence, $L_{\cD_B}$ is contained in the lines joining a point of
$L_{\cD_B}\cap r_1$ with a point of $L_{\cD_B}\cap r_1^\perp$, where
$$L_{\cD_B}\cap r_1=\{\langle(x-bx^{q^r},0,0,\xi
(x+bx^{q^r})\rangle\,:\, x\in\F_{q^n}^*\}$$ and $$L_{\cD_B}\cap
r_1^\perp=\{\langle(0,y-by^{q^r},y+by^{q^r},0)\rangle\,:\,
y\in\F_{q^n}^*\}$$ and both linear sets $L_{\cD_B}\cap r_1$ and
$L_{\cD_B}\cap r_1^\perp$ are of type described in Lemma
\ref{lemma:LineCASEB}.

Arguing as in the proof of Theorem \ref{thm:linearsetCaseA}, direct
computations show that $L_{\cD_B}$ is a maximum scattered
$\F_q$-linear set of $\Lambda = PG(3,q^n)$. Now, consider the lines
$$t_1:=\{\langle(x,y,y,\xi x)\rangle: \,\, x,y \in\F_{q^n}, (x,y)\neq
(0,0)\}$$ and $$t_2:=\{\langle(-x,-y,y,\xi x)\rangle: \,\, x,y
\in\F_{q^n}, (x,y)\neq (0,0)\}$$ of $\Lambda$.  Since $q$ is odd and
$\xi$ is a non--square in $\F_q^*$, these lines are disjoint,
external to the quadric $\cQ$ and pairwise polar with respect to the
polarity defined by $\cQ$. Consider the collineation $\Phi_f$
between $t_1$ and $t_2$ defined by the following strictly semilinear
map  $$f:  (x,y,y,\xi x) \mapsto  (-bx^{q^r},-by^{q^r},by^{q^r},\xi
bx^{q^r}), $$ where $\xi$ and $b$ are as in the definition of the
$\F_q$-linear set $L_{\cD_B}.$ Then we have $$L_{\cD_B} = \{\langle
(x,y,y,\xi x)+f((x,y,y,\xi x)) \rangle \,:\, x,y \in \F_{q^n}
\,\text{and}\, (x,y)\neq (0,0)\}.$$ Hence, by Theorem
\ref{thm:algebraicpseudoregulus}, $L_{\cD_B}$ is an $\F_q$-linear
set of $\Lambda$ of pseudoregulus type and the transversal lines of
the associated pseudoregulus are $t_1$ and $t_2$. This completes the
proof.
 \end{proof}

\subsection*{$\cD_{AB}$ presemifields}

In this section we will describe the linear set associated with the
presemifields $\cD_{AB}$, precisely:
$$L_{\cD_{AB}}=\{\langle(x,y,A_{b^2,r}(y),\xi B_{b,-r}(x))\rangle\, : \, x,y \in \F_{q^n}, (x,y) \neq (0,0)\} \subset PG(3,q^n),$$
where $A_{b^2,r}(x)=x^{q^r}-b^2x^{q^{-r}}$ and $B_{b,-r}(x)=2H_{b,-r}^{-1}(x)-x$, with $H_{b,-r}(x)=x-bx^{q^{n-r}}$, $gcd(n,r)=1$ and $b\in\F_{q^n}^*$ such that $N_q(b)\ne \pm 1$.\\

\begin{theorem}  \label{thm:linearsetCaseAB}
The $\F_q$-linear set $L_{\cD_{AB}}$ has the following structure:

\begin{itemize}
\item [$i)$] If $N_q(b^2) \neq -1$ and $n=3$, then $L_{\cD_{AB}}$ is a maximum
scattered $\F_q$-linear set of $PG(3,q^3)$ of pseudoregulus type.
\item [$ii)$] If $N_q(b^2) \neq -1$ and $n>3$, then $L_{\cD_{AB}}$ is a maximum
scattered $\F_q$-linear set of $PG(3,q^n)$ with two long lines one
of pseudoregulus type and the other one not of
pseudoregulus type.
\item [$iii)$] If $N_q(b^2)= -1$, then the points of $L_{\cD_{AB}}$ of weight
greater than one have weight $2$, they are $(q^{n-1}-1)/(q^2-1)$ in number. Hence $|L_{\cD_{AB}}|=\frac{q^{2n+1}+q^{2n}-q^{n}-1}{q^2-1}$.
\end{itemize}
\end{theorem}
\begin{proof}
Observe that $L_{\cD_{AB}}$ is contained in the lines joining a
point of $L_{\cD_{AB}}\cap r_1$ with a point of $L_{\cD_{AB}}\cap
r_1^\perp$, where
$$L_{\cD_{AB}}\cap r_1=\{\langle(x,0,0,\xi B_{b,-r}(x))\rangle\,:\,
x\in\F_{q^n}^*\} \,\,\,\text{and}\,\,\, L_{\cD_{AB}}\cap
r_1^\perp=\{\langle(0,x,A_{b^2,r}(x),0)\rangle\,:\,
x\in\F_{q^n}^*\}.$$ The linear sets $L_{\cD_{AB}}\cap r_1$ and
$L_{\cD_{AB}}\cap r_1^\perp$ satisfy the properties described in
Lemmas \ref{lemma:LineCASEB} and \ref{lemma:LineCASEA},
respectively. By these two lemmas and by Theorems
\ref{thm:linearsetCaseA} and \ref{thm:linearsetCaseB}, arguing as in
the proof of Theorem \ref{thm:linearsetCaseA}, can be easily checked
that a point of $L_{\cD_{AB}}$ of weight greater than 1 exists if
and only if $N_q(b^2)=-1$, and in such a case it belongs to the line
$r_1^\perp$ and it has weight 2. So, in Case $i)$, $L_{\cD_{AB}}$ is
a maximum scattered $\F_q$-linear set of $PG(3,q^3)$ and, hence, by
\cite[Proposition 2.8]{MPT}, it is of pseudoregulus type. Point
$ii)$ of the statement also follows from Point $i)$ of Lemma
\ref{lemma:LineCASEA}, from Lemma \ref{lemma:LineCASEB} and from
Theorems \ref{thm:linearsetCaseA} and \ref{thm:linearsetCaseB}.
Finally, Point $iii)$ may be easily achieved taking into account
Point $ii)$ of Lemma \ref{lemma:LineCASEA} and arguing as in the
proof of Point $iii)$ of Theorem \ref{thm:linearsetCaseA}.
\end{proof}

\subsection*{Proof of Theorem \ref{thm:main}}

Since each presemifield $\cD_{A}$, $\cD_{B}$ and $\cD_{AB}$ has
dimension $2n$ ($n$ odd) over its center, none of these
presemifields is isotopic to a ${\mathcal {EMPT}2}$ semifield. By
looking at the geometric structure of their associated linear sets,
one also get that none of the three relevant presemifields is
isotopic to a $\mathcal{GD}$ semifield with either $s$ or $t$ equal
zero, to a $\cal{TP}$ semifield or to a $\cal{TP}^{\perp}$
semifield.

{\bf CASE $\cD_A$}.\quad Comparing the nuclei of the presemifields $\cD_A$
(see Remark \ref{rem:nucleiDemp}) we get that none of these presemifields is
isotopic to a $\cK_{17}$, $\cK_{19}$, $\cGD$, ${\mathcal {JMPT}}$ and ${\mathcal {EMPT}1}$ semifield.

It remains to compare $\cD_A$ with a rank 2 Generalized Twisted
Field $\cA$. By \cite[Prop. 5.3]{LuMaPoTr-Sub} the $\F_q$--linear set associated with a Generalized Twisted Field is of pseudoregulus type. If $N_q(a)= -1$, by Theorem \ref{thm:linearsetCaseA},
the associated linear set $L_{\cD_A}$ is not scattered, whereas if
$N_q(a) \neq  -1$ and $n>3$, $L_{\cD_A}$ is a maximum
scattered linear set with two long lines not of pseudoregulus type,
hence, by Remark \ref{rem:examplenonpseudoregulustype}, $L_{\cD_A}$ is not of pseudoregulus type. Then, in both cases $\cD_A$ is not isotopic to a Generalized Twisted Field. So $\cD_A$
is new.\\

{\bf CASE $\cD_B$}.\quad By Theorem \ref{thm:linearsetCaseB} and by
\cite[Cor. 5.5]{LuMaPoTr-Sub} it follows that, for each $n\geq 3$, a presemifield
$\cD_B$ is isotopic to a rank 2 Generalized Twisted Field.\\

{\bf CASE $\cD_{AB }$}.\quad Comparing the nuclei of the presemifields $\cD_{AB}$
(see Remark \ref{rem:nucleiDemp}) we get that none of these presemifields is
isotopic to a $\cK_{17}$, $\cK_{19}$, $\cA$ and ${\mathcal {JMPT}}$ (pre)semifield.

If $N_q(b^2)\neq -1$ and $n>3$, it remains to compare $\cD_{AB}$ with a $\cGD$ semifield. By Theorem \ref{thm:linearsetCaseAB}, the associated linear set $L_{\cD_{AB}}$ is a maximum
scattered linear set with a long line not of pseudoregulus type; hence, by Remark \ref{rem:examplenonpseudoregulustype}, $L_{\cD_{AB}}$ is not of pseudoregulus type. So, by comparing it with the linear set of a $\cGD$ semifield (see Section 3 and Theorem \ref{thm:ScatGenDi}) we get that corresponding (pre)semifields cannot be isotopic.

If $N_q(b^2)=-1$ and $n>3$, it remains to compare $\cD_{AB}$ with an ${\mathcal {EMPT}1}$ semifield and with a non--scattered $\cGD$ semifield. From Section 3, the $\F_q$--linear set of an ${\mathcal {EMPT}1}$ semifield has at least $q+1$ points of weight greater than 2, and hence from $iii)$ of Theorem \ref{thm:linearsetCaseAB}, $\cD_{AB}$ cannot be isotopic to an ${\mathcal {EMPT}1}$ semifield.

Finally, if $\cD_{AB}$ were isotopic to a non--scattered $\cGD$ semifield, comparing the structure of the associated $\F_q$--linear set (see Theorem \ref{thm:linearsetCaseAB} and Section 3), we would have that the two integers $s$ and $t$ appearing in the multiplication of a $\cGD$ semifield should satisfy the arithmetic condition $\{\gcd(s,n),\gcd(t,n)\}=\{1,2\}$. This would imply $n$ even, which is not the case for a $\cD_{AB}$ semifield.

\bigskip
\noindent As we have observed at the end of Section $3$, all rank two (pre)semifields listed there have associated $\F_q$-linear set whose geometric structure is invariant under the transpose with the exception of $\cK_{17}$ and $\cK_{19}$ which are pairwise transpose. Hence, by the above
arguments, none of the semifields $\cD_A$ and $\cD_{AB}$ is isotopic to the transpose of a (pre)semifield in the above mentioned list.

Finally, as the families $\cD_A$, $\cD_B$ and $\cD_{AB}$ are closed under the translation dual operation (see Proposition \ref{prop:transl-dual}),
none of them is isotopic to the translation dual of a (pre)semifield in the above mentioned list.\qed

\subsection{Final Remark}
The isotopy issue for Dempwolff's presemifields  $\cD_A$ and
$\cD_{AB}$  with $n=3$ seems harder. This mainly because in this
case there are many more known examples in the literature to compare
with (see e.g. \cite{JoMaPoTr2008}, \cite{EbMaPoTr2009FFA},
\cite{JoMaPoTr2011}, \cite{LMPT}, \cite{MaPoTr2011}). For this
reason its study deserves a separate discussion based on a different
approach. This will be the main topic of a forthcoming paper by the
same authors.

\bigskip

\noindent {\bf Acknowledgement} The research of the first author is
supported by the Research Foundation Flanders-Belgium
(FWO-Vlaanderen). This work has been supported by a research project
from Universit\`a di Padova (Progetto di Ateneo CPDA113797/11), by
the Research Project of MIUR (Italian Office for University and
Research) ``Geometrie su Campi di Galois, piani di traslazione e
geometrie di incidenza'' and by the Research group GNSAGA of INDAM.

\bigskip
\bigskip
\bigskip

\noindent  Michel Lavrauw\\
Department of Management and Engineering,\\
Universit\`a di Padova,\\
I--\,36100 Vicenza, Italy\\
{\em michel.lavrauw@unipd.it}

\bigskip

\noindent Giuseppe Marino and Olga Polverino\\
Dipartimento di Matematica e Fisica,\\
 Seconda Universit\`a degli Studi
di Napoli,\\
I--\,81100 Caserta, Italy\\
{\em giuseppe.marino@unina2.it}, {\em olga.polverino@unina2.it}

\bigskip
\noindent  Rocco Trombetti\\
Dipartimento di Matematica e Applicazioni "R. Caccioppoli",\\
Universit\`a degli Studi di Napoli ``Federico II'',\\
 I--\,80126 Napoli, Italy\\
{\em rtrombet@unina.it}

\end{document}